\documentclass[a4paper, 11pt]{amsart}
\usepackage{amssymb,amsmath,amstext,amsfonts,amsthm,mathrsfs}
\usepackage{mathtools,mathdots}
\usepackage[utf8]{inputenc}
\usepackage{xcolor}
\usepackage{tikz-cd}
\usepackage[notref,notcite]{showkeys}

\newcommand{\CC}{\mathbb{C}}

\newcommand{\extp}{{\textstyle \bigwedge}} 

\usepackage{hyperref}
\hypersetup{
colorlinks=true,
linkcolor=blue,
filecolor=magenta, 
urlcolor=cyan,
}
\theoremstyle{plain}
\newtheorem*{theorem*}{Theorem}
\newtheorem{theorem}{Theorem}
\newtheorem{MainTheorem}{Theorem}

\numberwithin{theorem}{section}

\newtheorem{lemma}[theorem]{Lemma}
\newtheorem{corollary}[theorem]{Corollary}

\newtheorem{remark}[theorem]{Remark}

\theoremstyle{definition}

\newcommand{\C}{\mathbb{C}}

\newcommand{\PP}{\mathbb{P}}

\newcommand{\cI}{\mathcal{I}}

\DeclareMathOperator{\codim}{{codim}}
\DeclareMathOperator{\Hom}{Hom}

\DeclareMathOperator{\sing}{Sing}

\DeclareMathOperator{\ann}{Ann}
\DeclareMathOperator{\cont}{\lrcorner}

\newcommand{\into}{\hookrightarrow}
\newcommand{\onto}{\twoheadrightarrow}

\newcommand{\OO}{\mathcal{O}}
\newcommand{\OS}{\mathcal{O}_S}
\newcommand{\OX}{\mathcal{O}_X}
\newcommand{\lra}{\longrightarrow}

\newcommand{\p}[1]{{\mathbb{P}^{#1}}}
\newcommand{\opn}{{\mathcal O}_{\mathbb{P}^{n}}}
\newcommand{\pn}{{\mathbb{P}^{n}}}

\newcommand{\tpn}{{\rm T}{\mathbb{P}^{n}}}

\newcommand{\sF}{\mathscr{F}}
\newcommand{\sG}{\mathscr{G}}
\newcommand{\sI}{\mathscr{I}}

\newcommand{\tF}{{\rm T}_\mathscr{F}}

\newcommand{\nF}{{\rm N}_\mathscr{F}}

\newcommand{\tX}{{\rm T}X}

\newcommand{\del}[1]{{\frac{\partial}{\partial x_{#1}}}}
\newcommand{\lder}{\mathcal{L}}
\newcommand{\rad}{ {\rm rad}}

\title{Remarks on the second Chern class of a foliation}

\author[A. Muniz]{Alan Muniz}
\address{Departamento de Matem\'atica \\ Centro de Ci\^encias Exatas e da Natureza \\ Universidade Federal de Pernambuco \\ Recife - PE, CEP 50740-560, Brasil}
\email{alan.nmuniz@ufpe.br}

\date{December 31st, 2025}

\makeatletter
\@namedef{subjclassname@2020}{\textup{2020} Mathematics Subject Classification}
\makeatother


\begin{document}

\begin{abstract}
We bound the second Chern class of the tangent sheaf of a codimension-one foliation. Equivalently, we bound the degree of the pure codimension-two part of the singular scheme. In particular, for a degree-$d$ foliation on the projective space, the codimension-two part of its singular scheme must have degree at least $d+1$. Moreover, equality holds only for rational foliations of type $(1,d+1)$. These bounds involve counting an invariant related to first-order unfoldings of 2-dimensional foliated singularities.  
\end{abstract}

\maketitle

\section{Introduction}
A codimension-one distribution $\sF$ on a complex manifold $X$ is given by an exact sequence
\[
0 \lra \tF \lra \tX \lra \sI_Z\otimes \nF \lra 0
\]
where $\tF$ is the \emph{tangent sheaf}, $\nF$ is the \emph{normal bundle}, and $Z = \sing(\sF)$ is the \emph{singular scheme}. The tangent sheaf $\tF$ is reflexive of rank $\dim X -1$ and $\nF$ is a line bundle. By Frobenius' Theorem, $\sF$ defines a foliation if $\tF$ is involutive under the Lie bracket of vector fields. 

One approach to studying distributions is to investigate the sheaf-theoretic properties of $\tF$ and its interplay with the singular scheme; see, for instance, \cite{artigao, CJM-deg1, GJM-deg2,CCM-3folds}. An important question is how the Frobenius integrability condition impacts the invariants of a foliation. The most prominent result in this direction is the Baum-Bott theorem \cite{BB}, which (in particular) describes how the characteristic classes of $\sI_Z\otimes \nF$ localize along $Z$. For instance, if $\codim Z \geq 3$ one gets  $c_1(\nF)^2 = 0$. One also has that the codimension-two part $Z_2$ of $Z$ must be connected, provided $\nF$ is ample, see \cite{CMJS}. This may fail for non-integrable distributions, see \cite[Example 4.6]{M-pforms}. 

In \cite{GJM-deg2}, we observed that foliations of degree $\leq 2$ on $\p3$ could only have second Chern classes in a certain small range compared to distributions in general. Equivalently, the degree of the codimension-two part of the singular locus must be restricted. In this work, we confirm this expectation. We prove a bound for the second Chern class of a foliation $\tF$, when $X$ is a projective manifold. Then we specialize to projective spaces. Our first main result is the following general bound. 

\begin{MainTheorem}\label{thm:A}
 Let $X$ be a complex projective manifold with a very ample divisor $H$. Let $\sF$ be a foliation on $X$. Then we have the following inequality.
  \[
  \delta(\sF,H) + P(\sF) \geq \int_Xc_2(\tF)H^{n-2} \geq P(\sF)
  \]
 where $P(\sF) = -(n-2)\int_X \left(  c_1(\nF) +K_X +\frac{(n-1) }{2} H\right)H^{n-1}$.
\end{MainTheorem}

The number $\delta(\sF,H)$ is defined in \S\ref{sec:Main} as follows. Given a (closed) point $p \in S\subset X$, let $\omega$ be a 1-form defining the restriction $\sF|_S$ around $p$. Then $\delta(\sF|_S,p)$ is the colength of the ideal generated by the coefficients of $\omega$ and $d\omega$. Then $\delta(\sF,H)$ is the sum of $\delta(\sF|_S,p)$ for a general $S $ in the class  of $ H^{n-2}$. The local index $\delta(\sF|_S,p)$ counts the first-order unfoldings of the germ of foliation at $p$, see Lemma \ref{lem:unfol-delta}. Moreover, $\delta(\sF,H) = 0$ if and only if the non-Kupka singularities of $\sF$ have codimension at least $3$; if $n=3$ this is called \emph{generalized Kupka}.

For $X = \pn$ the projective space, we get a more precise result. In this case, $\delta(\sF,H) \leq d^2$, and equality holds only if $\sF$ has a rational first integral of type $(1,d+1)$. As usual, we write cohomology classes as integers. 

\begin{MainTheorem}\label{thm:B}
Let $\sF$ be a degree-$d$ codimension-one foliation on $\pn$. Then, 
\[
d+ 1 \leq \deg Z_2 \leq d^2+d+1
\]
or, equivalently, 
\[
d^2 + \frac{(n-2)(n-1-2d)}{2} \geq c_2(\tF) \geq \frac{(n-2)(n-1-2d)}{2}.
\]
Moreover, if $\deg Z_2 = d+1$ then $\sF$ is a rational foliation of type $(1,d+1)$.   
\end{MainTheorem}

For $n=3$, this result improves \cite[Proposição 2.6]{FassarellaThesis}, where the author is interested in the geometry of the Gauss map of a foliation. Fassarella's result was the starting point of this work. 

This article is organized as follows. In Section \ref{sec:Prelim} we discuss basic results on foliations. We prove several lemmas that will help us prove the main results. In Section \ref{sec:Main} we prove Theorems \ref{thm:A} and \ref{thm:B}, cf. Theorem \ref{thm:general} and Corollary \ref{cor:boundPN}, respectively. 


\section{Preliminaries} \label{sec:Prelim}
Let $X$ be a complex projective manifold of dimension $n$. A codimension-one \emph{distribution} $\sF$ on $X$ is defined by an exact sequence
\[
0 \lra \tF \lra \tX \stackrel{\omega}{\lra} \sI_Z\otimes \nF \lra 0
\]
The sheaf $\tF$ is reflexive of rank $n-1$, called the \emph{tangent sheaf} of $\sF$, the sheaf $\nF$ is a line bundle called the \emph{normal bundle} of $\sF$, the ideal sheaf $\sI_Z$ defines the \emph{singular scheme} of $\sF$, which is the vanishing locus of the twisted 1-form $\omega \in H^0(\Omega_X^1\otimes L)$. A codimension-one \emph{foliation} on $X$ is a distribution as above satisfying the integrability condition $[\tF,\tF] \subset \tF$ or, equivalently, $\omega \wedge d\omega = 0$. 

We discuss below some preliminary results needed in the subsequent sections. 

\subsection{Lifting to the affine cone} 
Let $H$ be a very ample divisor on $X$ defining an embedding $X\into \PP^N$, and let $C_X \subset \CC^{N+1}\setminus \{0\}$ be the affine cone of $X$. The radial vector field $\rad = \sum_{j=0}^{N} x_j \del{j}$ acts on $\CC^{N+1}$ preserving $C_X$. A twisted differential form $\theta \in H^0(\Omega_X^p(kH))$ can be regarded as a differential form on $C_X$  such that $\rad \cont \theta=0$ and $\lder_\rad \theta = k\theta$; where
\[
\lder_v\theta = v\cont d\theta + d(v\cont \theta)
\]
is the Lie derivative on $C_X$. This construction allows us to make sense of the exterior derivative of a twisted form on $X$ as a homogeneous differential form on $C_X$. We can also regard these forms as global sections of exterior powers of the affine cotangent bundles on $X$. 

The Chern class $c_1(H) \in H^{1,1}(X,\CC) = H^1(\Omega_X^1)$ defines a nontrivial extension 
\[
0 \lra \OX  \stackrel{\rad}{\lra} \widehat{\tX} \lra \tX \lra 0 .
\]
The vector bundle $\widehat{\tX}$ is called the \emph{affine tangent bundle} of $X$ with respect to $H$. For $X = \pn$ and $H$ the hyperplane class, this is precisely the Euler sequence. In particular, $\widehat{\tpn} = \opn(1)^{n+1}$. 

The \emph{affine cotangent bundle} is the dual $\widehat{\Omega}_X^1 :=  \widehat{\tX}^\vee$. We also denote $\widehat{\Omega}_X^p := \extp^p \widehat{\tX}^\vee$. We have the short exact sequences
\[
0 \lra \Omega_X^{p} \lra \widehat{\Omega}_X^p \stackrel{\rad}{\lra} \Omega_X^{p-1} \lra 0 
\]
By this construction, we identify homogeneous differential forms on $C_X$ with global sections of $\widehat{\Omega}_X^p(kH)$ for some $k\geq 0$.

For $1 \leq p \leq n$, the rank of $\widehat{\Omega}_X^p(kH)$ is strictly bigger than $n$. Thus, a general global section does not vanish anywhere. Suppose we are given $\theta \in H^0(\widehat{\Omega}_X^p(kH))$ such that $\rad\cont \theta \neq 0$ and whose vanishing locus $W$ is nonempty. One interesting question is whether there is an optimal bound for the degree of $W$. If $W$ is zero-dimensional, we have the following bound. 

\begin{lemma}\label{lem:bound-domega}
Let $\theta \in H^0(\widehat{\Omega}_X^p(kH))$ be a section whose vanishing locus is a zero-dimensional subscheme $W$. Then the length of $W$ is bounded
\[
\ell(W) \leq ((n+1-p+k)H+K_X)^n.
\]
\end{lemma}

\begin{proof}
First, $W$ defined by the image of $\theta^\vee \colon (\widehat{\Omega}_X^p(kH))^\vee \to \OX$. Moreover, we have the natural isomorphism $(\widehat{\Omega}_X^p)^\vee \cong \widehat{\Omega}_X^{n+1-p}(-K_X)$. Note that $\widehat{\tX} \hookrightarrow \OX(H)^{N+1}$ for $N = h^0(\OX(H))$; hence $\OX(-tH)^{N+1} \onto \widehat{\Omega}_X^t$. Then, 
\[
\OX((p-n-1-k)H-K_X)^{\binom{N+1}{p}} \onto \widehat{\Omega}_X^{n+1-p}(-kH-K_X ) \onto \sI_W 
\]
hence $\sI_W(K_X + (n+1-p+k)H)$ is globally generated. Localizing at a point of $W$, one sees that $h^0(\sI_W(K_X + (n+1-p+k)H))\geq \codim W$. Then $W$ is contained in a complete intersection of $n$ divisors of class $K_X + (n+1-p+k)H$, and the result follows by Bézout. 
\end{proof}

\begin{remark}
For $X = \pn$ and $H$ the hyperplane class, one gets $\ell(W) \leq (k-p)^n$. This is expected since $\widehat{\Omega}_{\pn}^p(k) \cong \opn(k-p)^{\binom{n+1}{p}}$. However, for other varieties, this bound is less effective. Fiz $X \subset \p3$ a hypersurface of degree $l$, and $H$ the hyperplane class. For $\theta \in H^0(\widehat{\Omega}_X^2(kH))$, $W$ is contained in the zero locus of $\rad \cont \theta$, which is bounded by $c_2(\Omega^{1}_X(kH))$. Then one gets
\[
\int_X c_2(\Omega^{1}_X(kH)) - ((1+k)H+K_X)^2 = k(2-l) + 2l-3. 
\]
Thus, the top Chern class gives a better bound for $l\geq 3$. We wonder whether there is an optimal bound for $\ell(W)$ (in the zero-dimensional case) on other manifolds besides the projective spaces.   
\end{remark}

Our main interest is to study $d\omega$ for a given $\omega \in H^0(\Omega_X^p(kH))$, in particular, its vanishing locus. The following lemma is a consequence of the Euler relation for homogeneous polynomials.

\begin{lemma}
Let $\omega \in H^0(\Omega_X^p(kH))$. Then the ideal sheaf of the vanishing locus of $d\omega \in H^0(\widehat{\Omega}_X^p(kH))$ is locally generated by the coefficients of $\omega$ and $d\omega$. 
\end{lemma}

\begin{proof}

First, we note that the module of differential forms on $C_X$ are obtained from polynomial differential forms on $\CC^{N+1}$ and modding out by the ideal of $X$ and its differentials. Hence, we only need to prove the result for polynomial homogeneous forms.

Let $\omega$ be a homogeneous polynomial $p$-form of total degree $k$ in the variables $x_0, \dots, x_N$. We will restrict it to the affine chart $\{x_0 = 1\}$ and compare the ideals generated by the coefficients before and after restriction. We need to fix some notation.

Given $I = \{j_1 < \dots < j_r \} \subset \{0, \dots, N\}$, we use the multi-index notation $dx_I = dx_{j_1}\wedge \dots \wedge dx_{j_r}$. For $i\not\in I$, define the sign $dx_i \wedge dx_I = (-1)^{\sigma(i,I)} dx_{I\cup \{i\}}$. Then we write $\omega = \sum_{|I| = p} A_Idx_I$, for $A_I$ homogeneous polynomials of degree $k-p$. The equation $\rad \cont \omega = 0$ translates to: 
\begin{equation}\label{eq:euler-relation}
    \forall J \mid |J| = p-1, \quad \sum_{j\not\in  J} x_j A_{J\cup \{j\}}(-1)^{\sigma(j,J)} = 0.
\end{equation}
Moreover,
\begin{equation}
    d\omega = \sum_{|J| = p+1} \left(\sum_{i\in J} \frac{\partial A_{J\setminus \{i\} }}{\partial x_i} (-1)^{\sigma(i,J\setminus \{i\})} \right) dx_J. 
\end{equation}

Next, we restrict to the affine chart $\{x_0 = 1\}$ and compare the ideals defined by $d\omega$. Let $a_J = A_J|_{x_0=1}$. Note that restricting $\omega$ and $d\omega$ to $\{x_0 = 1\}$ eliminates every term with a $dx_0$. Let $\cI\in \CC[x_1, \dots, x_N]$ denote the ideal generated by the coefficients of $d\omega$ modulo $(x_0 -1)$, and let $\cI_0$ be the ideal generated by the coefficients of $d\omega|_{x_0=1}$. Then
\begin{align*}
    \cI_0 &= \left(\sum_{i\in J} \frac{\partial a_{J\setminus \{i\} }}{\partial x_i} (-1)^{\sigma(i,J\setminus \{i\})}  \,\Big|\, |J| = p+1 , \, 0 \not\in J  \right), \\
    \cI &= \cI_0 + \left(\frac{\partial A_{J\setminus \{0\} }}{\partial x_0}\Big|_{x_0= 1} + \sum_{i\in J \setminus \{0\}} \frac{\partial a_{J\setminus \{i\} }}{\partial x_i} (-1)^{\sigma(i,J\setminus \{i\})} \,\Big|\, 0 \in J  \right).
\end{align*}
We want to show that $\cI = \cI_0 + (a_K \mid 0\not\in K)$. Fix $K\in \{1, \dots , N\}$ such that $|K|=p$. We will show that the term in display above for $J = \{0\}\cup K$ is of the form $k\,a_K + \cI_0$. By the Euler relation and the contraction equation \eqref{eq:euler-relation}, one gets
\begin{align*}
    \frac{\partial A_{K}}{\partial x_0}\Big|_{x_0= 1} & = (k-p)a_K -\sum_{j=1}^N x_j\frac{\partial a_{K}}{\partial x_j} ,\\
    \frac{\partial a_{\{0\} \cup K  \setminus \{i\} }}{\partial x_i} &= -a_K(-1)^{\sigma(i,K \setminus \{i\})} - \sum_{j \not\in \{0\} \cup K\setminus \{i\}} x_j\frac{\partial a_{\{j\} \cup K\setminus \{i\}}}{\partial x_i}(-1)^{\sigma(j,K\setminus \{i\})}.
\end{align*}
Therefore, combining these equations, we have
\begin{align*}
 &\frac{\partial A_{K}}{\partial x_0}\Big|_{x_0= 1} + \sum_{i\in K} \frac{\partial a_{\{0\} \cup K  \setminus \{i\} }}{\partial x_i} (-1)^{\sigma(i,\{0\} \cup K\setminus \{i\})} = \\
  &\quad     = k\, a_K - \sum_{j \not\in \{0\} \cup K}(-1)^{\sigma(j,K)} x_j\sum_{i\in K\cup \{j\}} \frac{\partial a_{\{j\} \cup K\setminus \{i\}}}{\partial x_i}(-1)^{\sigma(i,\{j\}\cup K\setminus \{i\})} \\
&\quad = k\, a_K + \cI_0,
\end{align*}
which concludes the proof.
\end{proof}

\subsection{Chern classes of distributions}
Let $\sF$ be a codimension-one distribution on a complex projective manifold $X$. Considering the exact sequence
\[
0 \lra \tF \lra \tX \stackrel{\omega}{\lra} \sI_Z\otimes \nF \lra 0
\]
one sees that the Chern classes of $\tF$ are related to the invariants of $Z$. Indeed, one gets an immediate equation for the Chern characters
\[
ch(\OO_Z) = ch(\OX) - (ch(\tX) -ch(\tF))ch(\nF^\vee)
\]
Let $Z_2$ denote the subscheme of $Z$ of pure codimension two. The degree of $Z_2$ with respect to an ample divisor $H$ is the intersection product
\[
\deg_H(Z_2) = \int_X [Z_2] H^{n-2} 
\]
which coincides with the leading coefficient of the Hilbert polynomial of $\OO_Z$ times $(n-2)!$. Note that the lower-dimensional components of $Z$ do not affect the leading coefficient. By the Hirzebruch-Riemann-Roch Theorem, one gets $\deg_H(Z_2) = \int_Xch_2(\OO_Z)H^{n-2}$. Using the fact that the Chern character is additive, one further gets
\begin{align*}
  \deg_H(Z_2) 
  & = \int_X (ch_2(\tF) - ch_2(\tX) + ch_2(\nF))H^{n-2}
\end{align*}
We may then rewrite it to obtain
\begin{multline}\label{eq:c2tf}
    \int_X c_2(\tF) H^{n-2} =  \int_X (K_X c_1(\nF)+ c_2(\tX) + c_1(\nF)^2 )H^{n-2}+ \\ -  \deg_H(Z_2) .  
\end{multline}

The degree of $Z_2$ is bounded above by the number of singularities of $\sF$ restricted to a general surface of class $H^{n-2}$. Then we get the following general inequality. 

\begin{lemma}\label{lem:lower-bound}
Let $\sF$ be a codimension-one distribution on a complex projective manifold $X$, and let $H$ be a very ample divisor. Then
\begin{equation}\label{eq:bound-C2}
\int_X c_2(\tF)H^{n-2} \geq -(n-2)\int_X \Bigg(  c_1(\nF) +K_X +\frac{(n-1) }{2} H\Bigg)H^{n-1}.    
\end{equation}
\end{lemma}

One could ask that $H$ be only ample and work with $aH$ for some $a \gg 0$. The only problem is that the right-hand side is not homogeneous on $a$, yielding a weaker result.

\begin{proof}
Let $S$ be a general surface in the class $H^{n-2}$ transverse to $\sF$. Then $Z_2\cap S \subset \sing(\sF|_S)$. In particular,
\[
\deg_H(Z_2) = \int_S [Z_2 \cap S] \leq \ell(\sing(\sF|_S)) = \int_S c_2(\Omega^1_S\otimes \nF).   
\]
On the other hand, the conormal sequence for $S$ is 
\[
0 \lra \OS(-H)^{n-2} \lra \Omega_X^1|_S \lra \Omega_S^1 \lra 0 .
\]
By adjunction, $K_S = (K_X + (n-2)H)|_S$.  Then,
{
\begin{align*}
    \int_S c_2(\Omega_S^1\otimes \nF) & = \int_S \frac{c_1(\Omega_S^1\otimes \nF)^2}{2} - ch_2(\Omega_S^1\otimes \nF) \\
    &  = \int_X \Bigg((c_1(\nF) +  (n-2) H + K_X)c_1(\nF) +  \\ & \qquad +c_2(\tX) + (n-2) K_X H +\frac{(n-2)(n-1) }{2} H^2\Bigg)H^{n-2}  . 
\end{align*}
}
Therefore, by equation \eqref{eq:c2tf} and the above discussion,
{
\begin{align*}
   \int_X c_2(\tF)H^{n-2} &\geq \int_X (K_X c_1(\nF)+ c_2(\tX) + c_1(\nF)^2 )H^{n-2} \\ & \qquad -\int_S c_2(\Omega_S^1\otimes \nF) \\
  &= -(n-2)\int_X \Bigg(  c_1(\nF) +K_X +\frac{(n-1) }{2} H\Bigg)H^{n-1} .
\end{align*}
}
\end{proof}

Recall that the discriminant of a coherent sheaf $F$ of rank $r$ is $\Delta(F) = 2r\,c_2(F) - (r-1)^2c_1(F)^2$. If $F$ is torsion-free and $\mu_H$-semistable, then it satisfies the Bogomolov inequality $\int_X \Delta(F)H^{n-2} \geq 0$. From Lemma \ref{lem:lower-bound} we deduce the following weaker inequality, without the semistability hypothesis:
\begin{equation}\label{eq:disc-bound}
     \int_X \Delta(\tF)H^{n-2} \geq -(n-2)\int_X (K_X+c_1(\nF)+(n-1)H)^2H^{n-2}.
\end{equation}

For $X = \pn$, $H$ the hyperplane class, and $c_1(\nF) = (d+2)H$, equation \eqref{eq:c2tf} and inequalities , \eqref{eq:bound-C2} and \eqref{eq:disc-bound} become
\begin{align}\label{eq:c2tfPN}
c_2(\tF)  &= d^2 + \frac{(n-3)(n-2d)}{2} + 2 - \deg(Z_2), \\
c_2(\tF) & \geq   \frac{(n-2)(n-1-2d)}{2} \\
\Delta(\tF) &\geq -(n-2)d^2.
\end{align}

For our purposes, we only need the second Chern class. However, it would also be interesting to describe the higher Chern classes of $\tF$ in terms of $Z$. For the case $X= \p3$, see \cite{artigao}. 

\subsection{A Local Invariant}\label{ssec:Loc}
Consider a germ of codimension-one foliation $\sG$ on $(\CC^2,0)$ defined by a germ of 1-form $\eta$, singular at the origin. Define the integer
\begin{equation}\label{eq:delta-local}
\delta(\sG,0) := \ell\left(  \frac{\OO_{\CC^2,0} }{ J(\eta, d\eta)} \right) = \dim_{\CC} \frac{\OO_{\CC^2,0} }{ J(\eta, d\eta)},     
\end{equation}
where $J(\eta, d\eta)$ denotes the ideal generated by the coefficients of $\eta$ and $d\eta$. 
This number does not depend on the choice of $\eta$ defining $\sG$ and is an analytic invariant.

The local invariant $\delta(\sG,0)$ is clearly bounded above by the Milnor number $\mu(\sG,0) = \ell\left({\OO_{\CC^2,0} }/{ J(\eta)} \right)$. Moreover, it is related to the first-order unfoldings of $\sG$. Following Suwa \cite{Suwa-unfoldings}, define 
\[
I(\eta) = \{\, h \in \OO_{\CC^2,0} \mid h d\eta = \eta \wedge \theta, \, \text{for some}\, \theta \in \Omega^1_{\CC^2,0} \,\}.
\]
This ideal is sometimes referred to as the ideal of \emph{persistent singularities}; see \cite{CMQ} and references therein. If $\eta = a\,dx + b\,dy$, then $d\eta = f \,dx\wedge dy$ where $f = \frac{\partial b}{\partial x} -\frac{\partial a}{\partial y}$. Hence,
\[
I(\eta) = (J(\eta) : (f) ) = ( J(\eta) : J(\eta,d\eta)).
\]
First-order unfoldings of $\sG$ are parameterized by the set $U(\sG) = I(\eta)/J(\eta)$, see \cite[\S4]{Suwa-unfoldings}. 

\begin{lemma}\label{lem:unfol-delta}
Let $\sG$ be a germ of foliation on $(\CC^2,0)$. Then, $\ell(U(\sG)) = \delta(\sG,0)$. 
\end{lemma}

\begin{proof}
If $\sG$ is regular, then $\ell(U(\sG)) = \delta(\sG,0) = 0$. Assume $\sG$ is singular. Consider the algebra $A = \OO_{\CC^2,0}/J(\eta)$. Since $a$ and $b$ have no common factors, $A$ is a complete intersection ring of dimension zero, hence Artinian Gorenstein. Moreover, $U(\sG)$ is a finite $A$-module, and we have
\[
U(\sG) = \ann_A(\overline{f}) = \Hom_A(A/(\overline{f}), A).
\]
Therefore, by Matlis' duality \cite[Proposition 3.2.12]{BH}, 
\[
\ell(U(\sG)) = \ell (\Hom_A(A/(\overline{f}), A)) = \ell (A/(\overline{f})) = \delta(\sG,0).
\]
\end{proof}

Foliations satisfying $U(\sG) = 0$ are then rigid, i.e., every unfolding is trivial, see \cite[\S6]{Suwa-unfoldings}. In light of the Lemma above, this means that either $\sG$ is regular, or $d\eta(0) \neq 0$, a Kupka singularity. On the other end, if $\delta(\sG,0) = \mu(\sG,0)$, then $U(\sG) = A$ and there exists a holomorphic 1-form $\theta$  such that $d\eta = \eta \wedge \theta$, or, equivalently, $\eta$ admits a first-order unfolding 
\[
\widetilde \eta  =  \eta + t\theta + dt
\]
corresponding to the class of $1 \in U(\sG)$. If this unfolding extends to an infinitesimal unfolding of $\eta$, then $\eta = gdf$ for some holomorphic functions $f,g$, see \cite[Proposition 6.14]{Suwa2}.

The invariant $\delta(\sG,0)$ plays a major role in our main result. The following is our key lemma.

\begin{lemma}\label{lem:key-lemma}
Let $\sF$ be a germ of codimension-one foliation on $(\CC^n,0)$ singular along a subscheme of pure codimension two $Z$. Let $(S, 0) \subset (\CC^n,0)$ be a germ of smooth surface transverse to $\sF$. Then the intersection multiplicity $(Z\cdot S)_0$ satisfies
\[
\mu(\sF|_S,0)  \geq (Z\cdot S)_0 \geq \mu(\sF|_S,0) - \delta(\sF|_S,0),
\]
where $\mu$ denotes the Milnor number.
\end{lemma}

\begin{proof}
    Let $\omega = \sum_{j=1}^n a_j dx_j \in \Omega_{\CC^n,0}^1$ a germ of integrable 1-form defining $\sF$. Up to a change of coordinates, we may suppose that $S = V(x_3, \dots, x_n)$ and denote by $\overline{a_j}$ the class of $a_j$ modulo $(x_3, \dots,x_n)$. Let $W = \sing(\sF|_W)$. Then 
    \[
    I_{(Z\cap S)} = (\overline{a_1}, \dots, \overline{a_n}) \quad \text{and} \quad  I_W = (\overline{a_1},\overline{a_2}). 
    \]
    The inequality $\mu(\sF|_S,0)  \geq (Z\cdot S)_0$ follows directly. 
    
    Consider the quotient ring $A = \OO_S/I_W  = \OO_{\CC^n,0}/(a_1,a_2,x_3, \dots, x_n)$ and the ideal $J = I_{(Z\cap S)}\cdot A = I_{(Z\cap S)}/I_W$. Then, from the short exact sequence
    \[
    0 \lra J \lra A \lra \OO_S/I_{(Z\cap S)} \lra 0
    \]
    one gets
    \[
    (Z\cdot S)_0 = \ell(\OO_S/I_{(Z\cap S)}) = \ell(\OO_S/I_W) -\ell(J) = \mu(\sF|_S,0) - \ell(J).
    \]
    To conclude, we only need to prove that $\ell(J) \leq \delta(\sF|_S,0)$. Note that $\sF|_S$ is defined by $\overline{\omega} = \overline{a_1} dx_1 + \overline{a_2}dx_2$ and $d\overline{\omega} = fdx_1\wedge dx_2$ for $f = \frac{\partial \overline{a_2} }{\partial x_1} - \frac{\partial \overline{a_1} }{\partial x_2}$. Then 
    \[
    \delta(\sF|_S,0) =  \ell \left(\frac{\OO_S}{\left(\overline{a_1},\overline{a_2}, f \right)}\right) = \ell \left(A/(f)\right).
    \]
    On the other hand, the integrability equation $\omega \wedge d\omega = 0$ translates to 
    \[
    a_k \left( \frac{\partial a_j }{\partial x_i} - \frac{\partial  a_i }{\partial x_j} \right) - 
    a_j \left( \frac{\partial a_k }{\partial x_i} - \frac{\partial  a_i }{\partial x_k} \right) + 
    a_i \left( \frac{\partial a_k }{\partial x_j} - \frac{\partial  a_j }{\partial x_k} \right) = 0, \quad \forall i<j<k.
    \]
   Therefore, for every $k\geq 3$ we have $\overline{a_k}f  = 0 \in A$. This means that $J \subset \ann_A(f)$; in particular, $\ell(J) \leq \ell(\ann_A(f))$. As in the proof of Lemma \ref{lem:unfol-delta}, we obtain 
   \[
   \ell(J) \leq \ell(\ann_A(f)) = \delta(\sF|_S,0).
   \]
\end{proof}

Note that $\sF$ can be seen as an unfolding of $\sF|_S$. If $\delta(\sF|_S,0) = 0$ then $\sF$ has a Kupka singularity. In particular, $Z$ must be smooth and $\sF$ is biholomorphic to a foliation given by a 1-form in two variables, a trivial unfolding. 

Following the proof of Lemma \ref{lem:key-lemma}, one sees that if $Z$ is a (local) complete intersection, then $\mu(\sF|_S,0)  = (Z\cdot S)_0$.

Finally, we remark that $\delta(\sG,0)$ can be defined, mutatis mutandis, for one-dimensional foliations on $(\CC^n,0)$ with an isolated singularity. Then, the corresponding version of Lemma \ref{lem:unfol-delta}, although the relation to unfolding is less clear. An analogue of Lemma \ref{lem:key-lemma} for this situation eludes us for the moment.

\section{Main Results}\label{sec:Main}
First, we define the global counterpart of the invariant $\delta$. Given a foliation $\sG$ on a surface $S$, define
\[
\delta(\sG) := \sum_{p\in \sing(\sG)} \delta(\sG,p). 
\]
Given a foliation $\sF$ on a complex projective manifold $X$ of dimension $n\geq 3$ with a very ample divisor $H$, define 
\[
\delta(\sF,H) := \delta(\sF|_S)
\]
for a generic surface $S$ of class $H^{n-2}$.

\begin{theorem}\label{thm:general}[Theorem \ref{thm:A}]
  Let $X$ be a complex projective manifold with a very ample divisor $H$. Let $\sF$ be a foliation on $X$. Then we have the following inequality.
  \[
  \delta(\sF,H) + P(\sF) \geq \int_Xc_2(\tF)H^{n-2} \geq P(\sF)
  \]
  where $P(\sF) = -(n-2)\int_X \left(  c_1(\nF) +K_X +\frac{(n-1) }{2} H\right)H^{n-1}$.
\end{theorem}

\begin{proof}
 Let $S$ be a general complete intersection surface in the class of $H^{n-2}$ such that $\sF$ is transversal to $S$. Thus $\sF|_S$ has normal bundle $\nF|_S$ and $S$ intersects $Z_2$ transversely and avoids lower-dimensional components of $Z$. By Lemma \ref{lem:key-lemma},
 \begin{equation}
    c_2(\Omega^1_S \otimes \nF) \geq \deg_H(\sing(\sF)_2) \geq c_2(\Omega^1_S \otimes \nF) - \delta(\sF,H)
\end{equation}
Then, by Lemma \ref{lem:lower-bound}, we get the lower bound. The upper bound follows from the same argument in Lemma \ref{lem:lower-bound} and the inequalities above. 
\end{proof}


For $X = \pn$, $H$ the hyperplane class and $\nF = \opn(d+2)$ we can get more. Indeed, in this case, we may apply Lemma \ref{lem:bound-domega}. 

\begin{corollary}\label{cor:boundPN}[Theorem \ref{thm:B}]
Let $\sF$ be a degree-$d$ codimension-one foliation on $\pn$. Then 
\[
 d+ 1 \leq \deg Z_2 \leq d^2+d+1
\]
or, equivalently, 
\[
d^2 + \frac{(n-2)(n-1-2d)}{2} \geq c_2(\tF) \geq \frac{(n-2)(n-1-2d)}{2}.
\]
Moreover, if $\deg Z_2 = d+1$ then $\sF$ is a rational foliation of type $(1,d+1)$.
\end{corollary}

The argument for constructing the rational first integral follows the proof of \cite[Proposi\c{c}\~ao 2.6]{FassarellaThesis}, which draws from \cite[Proposition 5.1]{LnP}. 

\begin{proof}
First, the inequalities follow from Theorem \ref{thm:general} and $\delta(\sF,H) \leq d^2$, due to Lemma \ref{lem:bound-domega}. 

Let $L$ be a generic $2$-plane in general position with respect to $\sF$, and $\omega$ be a homogenous $1$-form defining $\sF|_L$. We can write 
\[
     d\omega = Adx_1\wedge dx_2 - B dx_0\wedge dx_3 + Cdx_0 \wedge dx_1,
\]
where $A, B, C$ have degree $d$. If $\deg Z_2 = d+1$, then $\delta(\sF|L) = d^2$.
Noether's AF+BG Theorem then yields a nontrivial relation $\alpha A + \beta B + \gamma C = 0$ with $\alpha,\beta,\gamma \in \C$, not all zero. Define $E= \alpha x_0 + \beta x_1 + \gamma x_2$. Then $d\omega \wedge dE = 0$, hence $d\omega = \theta \wedge dE $ for some homogeneous $1$-form $\theta$. Contracting with the radial vector field, we get
\[
(d+2)\omega = \rad \cont d\omega = (\rad \cont \theta) dE - E\theta.
\]
Then the rational $1$-form $\frac{\omega}{E^{d+2}}$ is closed. Applying \cite[Lemma 9]{CLS}, there exists a closed rational $1$-form $\eta$ extending $\frac{\omega}{E^{d+2}}$ and defining $\sF$ outside its polar divisor $(\eta)_\infty$. Since $L$ is generic, $(\eta)_\infty$ is supported on a hyperplane $H$. Moreover, $\eta|_{\pn \setminus H}$ is exact and regular. Hence, $\sF$ has a first integral of the form $F/H^{d+1}$.     
\end{proof}

\bibliographystyle{abbrvurl}
\bibliography{biblio}

@phdthesis{FassarellaThesis,
  title={Sobre a aplicacao de Gauss de folheacoes holomorfas em espacos projetivos},
  author={do Amaral, Thiago Fassarella},
  year={2008},
  school={IMPA},
url = {https://impa.br/wp-content/uploads/2017/08/tese_dout_thiago_fassarella_amaral.pdf}
}

@article {CLS,
    AUTHOR = {Camacho, C. and Lins Neto, A. and Sad, P.},
     TITLE = {Foliations with algebraic limit sets},
   JOURNAL = {Ann. of Math. (2)},
  FJOURNAL = {Annals of Mathematics. Second Series},
    VOLUME = {136},
      YEAR = {1992},
    NUMBER = {2},
     PAGES = {429--446},
      ISSN = {0003-486X,1939-8980},
   MRCLASS = {32L30},
  MRNUMBER = {1185124},
MRREVIEWER = {Dominique\ Cerveau},
       DOI = {10.2307/2946610},
       URL = {https://doi.org/10.2307/2946610},
}

@article {LnP,
    AUTHOR = {Lins Neto, A. and Pereira, J. V.},
     TITLE = {The generic rank of the {B}aum-{B}ott map for foliations of
              the projective plane},
   JOURNAL = {Compos. Math.},
  FJOURNAL = {Compositio Mathematica},
    VOLUME = {142},
      YEAR = {2006},
    NUMBER = {6},
     PAGES = {1549--1586},
      ISSN = {0010-437X,1570-5846},
   MRCLASS = {32S65 (37F75)},
  MRNUMBER = {2278760},
       DOI = {10.1112/S0010437X06002326},
       URL = {https://doi.org/10.1112/S0010437X06002326},
}

@book {BH,
    AUTHOR = {Bruns, Winfried and Herzog, J\"urgen},
     TITLE = {Cohen-{M}acaulay rings},
    SERIES = {Cambridge Studies in Advanced Mathematics},
    VOLUME = {39},
 PUBLISHER = {Cambridge University Press, Cambridge},
      YEAR = {1993},
     PAGES = {xii+403},
      ISBN = {0-521-41068-1},
   MRCLASS = {13H10 (13-02)},
  MRNUMBER = {1251956},
MRREVIEWER = {Matthew\ Miller},
}

@article{artigao,
 author = {Calvo-Andrade, Omegar and Corr{\^e}a, Maur{\'{\i}}cio and Jardim, Marcos},
 title = {Codimension one holomorphic distributions on the projective three-space},
 fjournal = {IMRN. International Mathematics Research Notices},
 journal = {Int. Math. Res. Not.},
 issn = {1073-7928},
 volume = {2020},
 number = {23},
 pages = {9011--9074},
 year = {2020},
 language = {English},
 doi = {10.1093/imrn/rny251},
 zbMATH = {7323410},
 Zbl = {1460.32040}
}

@article{CJM-deg1,
 author = {Corr{\^e}a, Maur{\'{\i}}cio and Jardim, Marcos and Muniz, Alan},
 title = {Moduli of distributions via singular schemes},
 fjournal = {Mathematische Zeitschrift},
 journal = {Math. Z.},
 issn = {0025-5874},
 volume = {301},
 number = {3},
 pages = {2709--2731},
 year = {2022},
 language = {English},
 doi = {10.1007/s00209-022-03001-y},
 zbMATH = {7540302},
 Zbl = {1535.14023}
}

@article{GJM-deg2,
 author = {Galeano, Hugo and Jardim, Marcos and Muniz, Alan},
 title = {Codimension one distributions of degree 2 on the three-dimensional projective space},
 fjournal = {Journal of Pure and Applied Algebra},
 journal = {J. Pure Appl. Algebra},
 issn = {0022-4049},
 volume = {226},
 number = {2},
 pages = {32},
 note = {Id/No 106840},
 year = {2022},
 language = {English},
 doi = {10.1016/j.jpaa.2021.106840},
 zbMATH = {7389886},
 Zbl = {1470.14035}
}

@article{CCM-3folds,
 author = {Cavalcante, Alana and Corr{\^e}a, Mauricio and Marchesi, Simone},
 title = {On holomorphic distributions on {Fano} threefolds},
 fjournal = {Journal of Pure and Applied Algebra},
 journal = {J. Pure Appl. Algebra},
 issn = {0022-4049},
 volume = {224},
 number = {6},
 pages = {20},
 note = {Id/No 106272},
 year = {2020},
 language = {English},
 doi = {10.1016/j.jpaa.2019.106272},
 url = {hdl.handle.net/1843/EABA-AXHJ22},
 zbMATH = {7173746},
 Zbl = {1439.57047}
}

@article{BB,
 author = {Baum, Paul and Bott, Raoul},
 title = {Singularities of holomorphic foliations},
 fjournal = {Journal of Differential Geometry},
 journal = {J. Differ. Geom.},
 issn = {0022-040X},
 volume = {7},
 pages = {279--342},
 year = {1972},
 language = {English},
 doi = {10.4310/jdg/1214431158},
 zbMATH = {3423310},
 Zbl = {0268.57011}
}

@article{CMQ,
 author = {Calvo-Andrade, Omegar and Molinuevo, Ariel and Quallbrunn, Federico},
 title = {On the geometry of the singular locus of a codimension one foliation in {{$\mathbb{P}^n$}}},
 fjournal = {Revista Matem{\'a}tica Iberoamericana},
 journal = {Rev. Mat. Iberoam.},
 issn = {0213-2230},
 volume = {35},
 number = {3},
 pages = {857--876},
 year = {2019},
 language = {English},
 doi = {10.4171/rmi/1073},
 zbMATH = {7083115},
 Zbl = {1418.13011}
}

@misc{CMJS,
 author = {Calvo-Andrade, Omegar and Corr{\^e}a, Maur{\'{\i}}cio and Jardim, Marcos and Seade, Jos{\'e}},
 title = {On the connectedness of the singular set of holomorphic foliations},
 year = {2025},
 howpublished = {Preprint, {arXiv}:2506.08942 [math.{AG}] (2025)},
 url = {https://arxiv.org/abs/2506.08942},
 arXiv = {arXiv:2506.08942}
}

@incollection{Suwa-unfoldings,
 author = {Suwa, Tatsuo},
 title = {Unfoldings of codimension one complex analytic foliation singularities},
 booktitle = {Singularity theory. Proceedings of the symposium, Trieste, Italy, August 19-- September 6, 1991},
 isbn = {981-02-2000-6},
 pages = {817--865},
 year = {1995},
 publisher = {Singapore: World Scientific},
 language = {English},
 zbMATH = {1466303},
 Zbl = {0944.32022}
}

@article{Suwa2,
 author = {Suwa, Tatsuo},
 title = {Unfoldings of complex analytic foliations with singularities},
 fjournal = {Japanese Journal of Mathematics. New Series},
 journal = {Jpn. J. Math., New Ser.},
 issn = {0289-2316},
 volume = {9},
 pages = {181--206},
 year = {1983},
 language = {English},
 zbMATH = {3948724},
 Zbl = {0591.32020}
}

@article{M-pforms,
  doi = {10.13137/2464-8728/37300},
  url = {https://www.openstarts.units.it/handle/10077/37300},
  author = {Muniz,  Alan},
  language = {en},
  title = {p-Forms from Syzygies},
  publisher = {EUT Edizioni Università di Trieste},
  year = {2025}
}

\end{document}